\documentclass[reqno, 12pt]{article}

\pdfoutput=1

\usepackage{enumerate}
\usepackage{latexsym}
\usepackage[centertags]{amsmath}
\usepackage{amsfonts}
\usepackage{amssymb}
\usepackage{amsthm}
\usepackage{newlfont}
\usepackage{graphics}
\usepackage{color}
\usepackage{float}
\usepackage{diagbox}
\usepackage{tocloft}
\usepackage{titlesec}
\usepackage{booktabs}
\usepackage[pagebackref]{hyperref}
\usepackage[linesnumbered,ruled,vlined]{algorithm2e}
\usepackage{url} 
\usepackage{hyperref}
\usepackage{longtable}
\usepackage{rotating}
\usepackage{multirow}
\usepackage{extarrows}
\usepackage[sort,compress,numbers]{natbib}
\usepackage[utf8]{inputenc}

\newtheorem{thm}{Theorem}[section]

\newtheorem{prop}[thm]{Proposition}

\theoremstyle{mydefinition}
\newtheorem{dfn}[thm]{Definition}
\theoremstyle{myremark}

\allowdisplaybreaks[4]

\SetKwInput{KwInput}{Input}                % Set the Input
\SetKwInput{KwOutput}{Output}              % set the Output

\title{Proofs of four generating function conjectures for arbor polytopes}

\author{Feihu Liu$^{\color{blue} \dag}$ and Jinlong Tang$^{\color{blue} \S}$
\\[2mm]
{\small $^{\color{blue} \dag, \S}$ School of Mathematical Sciences,}\\[-0.8ex]
{\small Capital Normal University, Beijing, 100048, P.R.~China}\\
{\small {\color{blue} $^\dag$} Email address: liufeihu7476@163.com}\\
{\small {\color{blue} $^\S$} Email address: jinlong\_tang@cnu.edu.cn}\\
}

\date{\today}

\begin{document}

\maketitle

\begin{abstract}
This paper proves four conjectured generating series, due to Chapoton, which concern invariants of posets and polytopes associated with a specific sequence of arbors. Two of these conjectures provide closed-form formulas for the generating series of the Zeta polynomial and the generating series of the M-triangle of the poset, respectively. The remaining two conjectures pertain, respectively, to the Ehrhart polynomial and the Laplace transform of the volume function of the associated arbor polytope.
\end{abstract}

\noindent
\begin{small}
\end{small}

\noindent
\begin{small}
\end{small}

%\tableofcontents

\section{Introduction}

The combinatorial framework of \emph{arbors} was recently introduced to link rooted trees with geometric and algebraic structures. Arbor polytopes were first introduced and studied by Chapoton \cite{Chapoton}. Subsequently, the combinatorial properties of two special classes of arbor polytopes (i.e., the composition polytopes \cite{Athanasiadis-Ann} and the polytopes $Q_{n,k}$ \cite{Athanasiadis-Xiao-Yan}) were investigated.

An \emph{arbor} $t$ is defined as a rooted tree where each vertex $v$ is decorated by a non-empty subset of a finite set $[n]:=\{1,2,\ldots,n\}$, and these subsets form a partition of $[n]$. We call $n$ the \emph{size} of the arbor. The cardinality of a vertex $v$ is denoted by $|v|$.
For example, Figure \ref{Arbor1} is an arbor on the set $[8]$. 
\begin{figure}[htp]
\centering
\includegraphics[width=0.3\linewidth]{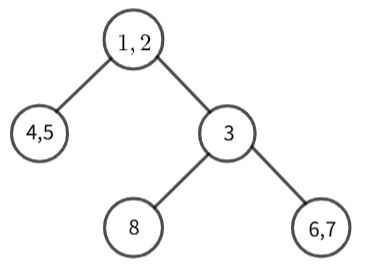}
\caption{An arbor on the set $[8]$.}
\label{Arbor1}
\end{figure}

To every arbor $t$ on a set $[n]$, we associate a lattice polytope $Q_t$ and a partial order $P_t$.

\begin{dfn}[The Polytope $Q_t$]
The polytope $Q_t$ is defined in the vector space $\mathbb{R}^n$ with coordinates $(x_i)_{i \in [n]}$ by the inequalities $x_i \ge 0$ for all $i \in [n]$ and 
\begin{align*}
\sum_{i \in \mathcal{D}(v)} x_i \le |\mathcal{D}(v)|
\end{align*}
for every vertex $v$, where $\mathcal{D}(v)$ is the set of labels in the sub-tree rooted at $v$.
\end{dfn}

For the example in Figure \ref{Arbor1}, besides all $x_i\geq 0$, the defining inequalities are 
$$x_8\leq 1, \quad x_6+x_7\leq 2, \quad x_3+x_6+x_7+x_8\leq 4, \quad x_4+x_5\leq 2,\quad x_1+x_2+\cdots +x_8\leq 8.$$

A convex polytope $\mathcal{P}$ is called \emph{lattice polytope} if all of its vertices have integer coordinates.
Chapoton \cite[Lemma 1.2]{Chapoton} proved that the arbor polytope is a simple lattice polytope.

\begin{dfn}[The Poset $P_t$]
The poset $P_t$ consists of the lattice points of $Q_t$ ordered by coordinate-wise comparison: $x \le y$ if $x_i \le y_i$ for all $i \in [n]$. This poset is graded by the height function $\mathrm{ht}(x) = \sum_{i \in [n]} x_i$.
\end{dfn}

The reader is referred to reference \cite[Section 3]{RP.Stanley} for the necessary knowledge on posets.
A standard reference on the polytope is \cite{BeckRobins}.
We are concerned with several classical invariants:
\begin{itemize}
    \item \textbf{Ehrhart Polynomial} $E_t(u)$: Counts the number of lattice points in the dilated polytope $uQ_t$.
    \item \textbf{Zeta Polynomial} $Z_t(u)$: Counts chains of length $u-2$ in the poset $P_t$.
    \item \textbf{M-triangle} $M_t(X, Y)$: A polynomial defined using M\"obius numbers and the rank function.
    \item \textbf{Laplace Transform of Volume} $L_t(v)$: Encodes the volume of slices of $Q_t$ according to height.
\end{itemize}

In this work, we focus on the specific arbors sequence $t_n$ ($n \ge 1$), which consists of a root vertex with one element, supporting $n-1$ child vertices that are each leaves with one element. Figure \ref{Arbor2} is the example of $t_5$. This sequence is closely related to Hochschild polytopes \cite{Combe21,Rivera-Saneb}.
The arbor polytope $Q_{t_n}$ is also a special case of a class of polytopes studied in the literature \cite{Athanasiadis-Xiao-Yan}.

\begin{figure}[htp]
\centering
\includegraphics[width=0.35\linewidth]{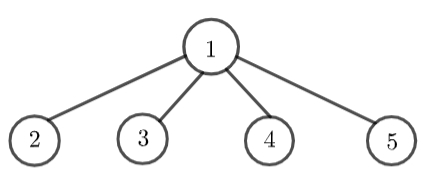}
\caption{The example of $t_5$.}
\label{Arbor2}
\end{figure}

Using \texttt{FriCAS} \cite{FriCAS}, Chapoton \cite[Section 9]{Chapoton} conjectured four formulas for various generating series in $s$ for the arbors $t_n$.
The four formulas respectively involve 
\begin{itemize}
    \item the Zeta polynomials of the posets $P_{t_n}$ (see Theorem \ref{Arbor-Theorem-1}),
    \item the M-triangles of the posets $P_{t_n}$ (see Theorem \ref{Arbor-Theorem-2}),
    \item the Ehrhart polynomials of the polytopes $Q_{t_n}$ (see Theorem \ref{Arbor-Theorem-3}),
    \item the Laplace transform of volume functions of the polyropes $Q_{t_n}$ (see Theorem \ref{Arbor-Theorem-4}).
\end{itemize}
In this paper, we mainly prove the four conjectures proposed by Chapoton, namely, the four theorems mentioned above.

The paper is organized as follows. Section \ref{Section-2-poset} introduces basic definitions on the Zeta polynomial and M-triangle of a poset. We then establish Theorems \ref{Arbor-Theorem-1} and \ref{Arbor-Theorem-2}.
In Section \ref{Section-3-polytopes}, we introduce the Ehrhart polynomial and the volume function of a lattice polytope. We then prove Theorems \ref{Arbor-Theorem-3} and \ref{Arbor-Theorem-4}.

\section{Generating Functions for Poset Invariants}\label{Section-2-poset}

\subsection{Zeta Polynomials}

Let $P$ be a finite poset. If $u \geq 2$, then define $Z_P(u)$ to be the number of chains $e_1 \leq e_2 \leq \cdots \leq e_{u-1}$ in $P$. We call $Z_P(u)$ (regarded as a function of $u$) the \emph{Zeta polynomial} of $P$. 
For the relevant properties of the Zeta polynomial, we refer the reader to \cite[Section 3.12]{RP.Stanley}

For an arbor $t$, Chapoton extends the Zeta polynomial of the poset $P_t$.
Let $Z_t(u,X)$ be the polynomial in $u$ and $X$ such that for every integer $m\geq 2$,
\begin{align*}
Z_t(m,X)=\sum_{e_1 \leq e_2 \leq \cdots \leq e_{m-1}} X^{\mathrm{ht}(e_{m-1})},
\end{align*}
where $\mathrm{ht}$ is the height of an element.
Setting $X=1$ yields the classical Zeta polynomial of the poset $P_t$. 

Let $t$ be an arbor on a set of cardinality $n\geq 2$. Let $(t^{(1)},\ldots,t^{(w)})$ be the sub-trees of the root vertex of $t$. Let $r$ be the size of the root vertex of $t$. 
Chapoton \cite{Chapoton} derived the following proposition.

\begin{prop}{\em \cite[Proposition 2.7]{Chapoton}}\label{Chap2.7}
Let $Z_{W}=\prod\limits_{k=1}^{w}Z_{t^{(k)}}(u,X)$.
Write $Z_{W}=\sum\limits_{j}W_{j}X^{j}$. Then for $0\le j\le n$, the coefficient of $X^{j}$ in $Z_{t}(u,X)$ is
$$ \sum_{l=\max(0,j-n+r)}^{j}\binom{r(u-1)+l-1}{l}W_{j-l}.$$
\end{prop}

The first equality in the following theorem was conjectured by Chapoton.
In fact, we obtain a simple closed formula during the proof.
\begin{thm}\label{Arbor-Theorem-1}
The generating series for the Zeta polynomials of $P_{t_n}$ is 
\[ 1 + \sum_{n \ge 1} Z_{t_n}(u, 1) s^n= \exp\left( \int \frac{u}{(1-us)(1+s-us)} \,ds \right)=\left( \frac{1-su+s}{1-su} \right)^u. \]
\end{thm}
\begin{proof}
For the arbor $t$ on one element, it is clear that the Zeta polynomial is $1+(u-1)X$.
Therefore, for $n-1$ leaves, we obtain
\begin{align*}
Z_W= (1+(u-1)X)^{n-1} = \sum_{k=0}^{n-1} W_k X^k, \quad \text{where} \quad W_k = \binom{n-1}{k}(u-1)^k.
\end{align*}
By Proposition \ref{Chap2.7}, for $0\leq j\leq n$, the coefficient of $X^{j}$ in $Z_{t_{n}}(u,X)$ is 
$$\sum\limits_{l=\max(0,j-n+1)}^{j}\binom{u-2+l}{l} W_{j-l}=\sum\limits_{l=0}^{min(j,n-1)}\binom{u-2+j-l}{j-l} W_{l}=\sum\limits_{l=0}^{j}\binom{u-2+j-l}{j-l} W_{l},$$
with the convention that $W_l= 0$ if $l\notin \{0,1,\ldots,n-1\}$.
Summing over all possible heights $0 \le j \le n$, we get
$$Z_{t_n}(u, 1) =\sum_{j=0}^n \sum\limits_{l=0}^{j}\binom{u-2+j-l}{j-l} W_{l}.$$
We extract the coefficients of $W_k$ ($0\leq k\leq n-1$) and obtain
$$\sum_{j=k}^n \binom{u+2+j-k}{j-k}=\sum_{l=0}^{n-k} \binom{u-2+l}{l}=\binom{u+n-k-1}{n-k}.$$
Therefore, we have
\begin{align*}
Z_{t_n}(u, 1)  &=\sum_{k=0}^{n-1} W_{k}\binom{u+n-k-1}{n-k}\\
&= \sum_{k=0}^{n-1} \binom{n-1}{k} (u-1)^k \binom{u+n-k-1}{n-k}\\
&=\sum_{k=0}^{\infty} \binom{n-1}{k} (u-1)^k \binom{u+n-k-1}{n-k}.
\end{align*}
Let $G(s) = 1 + \sum\limits_{n\geq1} Z_{t_n}(u, 1) s^n$. By making the substitution $m = n - k - 1$, and given that $n \geq 1$ and $k \geq 0$ and $\binom{n-1}{k}=0$ if $n-k-1<0$, the summation range for $m$ can be restricted to $[0,\infty)$ without altering the final sum. We get
\begin{align*}
G(s) &= 1+ \sum_{n\geq1} \sum_{k\geq0} \binom{n-1}{k} (u-1)^k \binom{u+n-k-1}{n-k} s^n\\
&= 1 + \sum_{m\geq0}\sum_{k\geq0} \binom{k+m}{k} (u-1)^k \binom{u+m}{m+1} s^{k+m+1} \\
&= 1 + \sum_{m\geq0} s^{m+1} \binom{u+m}{m+1} \left[ \sum_{k\geq0} \binom{k+m}{k} (s(u-1))^k \right] \\
&= 1 + \sum_{m\geq0} \binom{u+m}{m+1} \frac{s^{m+1}}{(1-su+s)^{m+1}}=\sum_{m\geq0} \binom{u+m-1}{m}\frac{s^m}{(1-su+s)^m}\\
&=\left(\frac{1}{1-\frac{s}{1-su+s}} \right)^u=\left( \frac{1-su+s}{1-su} \right)^u.
\end{align*}
Taking the natural logarithm and differentiating with respect to $s$:
\begin{align*}
\frac{d}{ds} \ln G(s) &= u \frac{d}{ds} \Big( \ln(1-us+s) - \ln(1-us) \Big)= \frac{u}{(1-us)(1+s-us)}.
\end{align*}
Integrating both sides and exponentiating (with the constant $C=0$ since $G(0)=1$) yields the desired result.
\end{proof}

\subsection{M-triangles}

Given a poset $P$, the \emph{M\"obius function} of $P$ (denoted $\mu_P$) is defined by 
\begin{align*}
\mu_P(a,a)&=1,\quad \text{for all} \quad a\in P;
\\ \mu_P(a,b)&=-\sum_{a\leq e<b} \mu_P(a,e), \quad \text{for all} \quad a<b \quad\text{in}\quad P.
\end{align*}
Given a graded poset $P$, one can define the \emph{$M$-triangle} of $P$ as the following sum of M\"obius numbers weighted using the rank function $\mathrm{ht}$ 
\begin{align*}
M_P(X, Y) = \sum_{a \leq b} \mu_P(a, b) X^{\mathrm{ht}(a)} Y^{\mathrm{ht}(b)}.
\end{align*}

For an arbor $t$, in order to obtain $M_{P_t}(X,Y)$, Chapoton introduced the polynomial
$$K_t(X,Y)=\sum_{a\in P_t}X^{\mathrm{nz}(a)}Y^{\mathrm{ht}(a)},$$ 
where $\mathrm{nz}(a)$ is the number of non-zero coordinates in $a$.
For the arbor $t$ on one element, it is clear that $K_t(X,Y)= 1 + XY$. 
Let $t$ be an arbor on $[n]$ with $n \geq 2$. Let $(t^{(1)},\ldots,t^{(w)})$ be the sub-trees of the root vertex of $t$. Let $r$ be the size of the root vertex of $t$. Chapoton obtained the following characterization of the coefficients of the polynomial $K_t(X,Y)$.

\begin{prop}{\em \cite[Proposition 2.8]{Chapoton}}\label{Chap2.8}
Let $W(X,Y)= \prod_{k=1}^w K_{t^{(k)}}(X,Y)$ and write $W(X,Y)= \sum_{i,j} W_{i,j} X^i Y^j$. Then for $0 \leq j \leq k \leq n$, the coefficient of $X^j Y^k$ in $K_t(X,Y)$ is
\[\sum_{\ell = \max(0,j-n+r)}^{\min(j,r)} \sum_{m = \max(\ell,k-n+r)}^{k+\ell-j} \binom{r}{\ell} \binom{m-1}{m-\ell} W_{j-\ell,k-m}. 
\]
\end{prop}

Furthermore, Chapoton derived a transformation formula form $K_t(X,Y)$ to $M_t(X,Y)$.
\begin{prop}{\em \cite[Proposition 2.10]{Chapoton}}\label{Chap2.10}
For every arbor $t$, we have
$$M_t(X,Y)=K_t(1-1/X, XY).$$
\end{prop}

In the particular case $t_{n}$ under consideration, we obtain a reduced formula.

\begin{prop}\label{FormulaK_t_n}
For any $n\in \mathbb{Z}_{+}$, the polynomial $K_{t_n}(X,Y)$ is given by
$$K_{t_n}(X,Y) = (1+XY)^n + \frac{XY^2}{1-Y} (1+XY)^{n-1} - \frac{XY^2}{1-Y} (Y(1+X))^{n-1}.$$
\end{prop}
\begin{proof}
For $0\le j\le k\le n$, we consider the coefficient of $X^{j}Y^{k}$ in $K_{t_{n}}(X,Y)$.

If $j=k$ which implies that the enumerated $a\in P_{t_{n}}$ possesses property $\mathrm{nz}(a)=\mathrm{ht}(a)$, so $a$ has exactly $j$ non-zero components, all of which are equal to $1$. One can easily see that there exist exactly $\binom{n}{j}$ such $a$. 

If $j<k$, it follows that the enumerated $a$ takes the value $k-j$ at the component of the root (which means $j\geq 1$), and among the remaining $n-1$ components, it possesses exactly $j-1$ non-zero entries, all of which are equal to $1$. There exist $\binom{n-1}{j-1}$ such $a$.

This yields the explicit expression for $K_{t_{n}}(X,Y)$,
\begin{align*}
 K_{t_n}(X,Y) &=\sum_{j=0}^{n}\binom{n}{j}X^{j}Y^{j}+\sum_{1\le j<k\le n}\binom{n-1}{j-1}X^{j}Y^{k}\\
 &= \sum_{j=0}^n \binom{n}{j} (XY)^j + \sum_{j=1}^{n-1} \binom{n-1}{j-1} X^j \sum_{k=j+1}^n Y^k \\
 &=(1+XY)^{n}+ \frac{XY^2}{1-Y}\sum_{j=1}^{n-1} \binom{n-1}{j-1} (XY)^{j-1} -\frac{XY^{n+1}}{1-Y}\sum_{j=1}^{n-1} \binom{n-1}{j-1} X^{j-1}\\
 &= (1+XY)^n + \frac{XY^2}{1-Y} (1+XY)^{n-1} - \frac{XY^2}{1-Y} (Y(1+X))^{n-1}.
\end{align*}
This completes the proof.
\end{proof}

\begin{thm}{\em (Conjectured by Chapoton)}\label{Arbor-Theorem-2}
The generating series for the M-triangles of $P_{t_n}$ is 
\[ 1 + \sum_{n \ge 1} M_{t_n}(X,Y) s^n = \frac{(XYs - Ys - 1)(XYs - 1)}{(2XYs - Ys - 1)(XYs - Ys + s - 1)}. \]
\end{thm}
\begin{proof}
By Proposition \ref{FormulaK_t_n} and  Proposition \ref{Chap2.10}, we have
\begin{align*}
&M_{t_n}(X,Y)= K_{t_n}(1 - 1/X, XY)\\
=& (1+XY-Y)^n + \frac{X^{2}Y^{2}-XY^{2}}{1-XY} (1+XY-Y)^{n-1} - \frac{X^{2}Y^{2}-XY^{2}}{1-XY} (2XY-Y)^{n-1}.
\end{align*}
Defining the substituted variables 
$$A = 1+XY-Y,\quad B = 2XY-Y,\quad \text{and}\quad C = \frac{X^2Y^2-XY^2}{1-XY},$$
we get 
$$M_{t_n}(X,Y) = A^n + C A^{n-1} - C B^{n-1}.$$
Now, by computing the generating function $G(s) = 1 + \sum\limits_{n\geq1} M_{t_n}(X,Y) s^n$, we obtain
\begin{align*}
G(s) &= 1 + \sum_{n\geq1} (As)^n + Cs \sum_{n\geq1} A^{n-1}s^{n-1} - Cs \sum_{n\geq1} B^{n-1}s^{n-1} \\
&= 1 + \frac{As}{1-As} + \frac{Cs}{1-As} - \frac{Cs}{1-Bs} \\
&=\frac{(XYs - Ys - 1)(XYs - 1)}{(2XYs - Ys - 1)(XYs - Ys + s - 1)}.
\end{align*}
This completes the proof.
\end{proof}

\section{Generating Functions for Polytope Invariants}\label{Section-3-polytopes}

\subsection{Ehrhart Polynomials}

Let $\mathcal{P}$ be a $d$-dimensional lattice convex polytope in $\mathbb{R}^n$.  
The \emph{Ehrhart function}  
$$E_{\mathcal{P}}(u)=|u\mathcal{P}\cap \mathbb{Z}^n|,\quad u=1,2,\ldots,$$
counts the number of integer points in the $u$-th dilation $u\mathcal{P} = \{ u\alpha : \alpha \in \mathcal{P} \}$ of $\mathcal{P}$.

Ehrhart's famous theorem \cite{Ehrhart62} says that $E_{\mathcal{P}}(u)$ is a polynomial in $u$ of degree $d$, now known as the \emph{Ehrhart polynomial} of $\mathcal{P}$. 
It is well-known \cite[Corollary 3.20; Theorem 5.6]{BeckRobins} that the leading coefficient of $E_{\mathcal{P}}(u)$ equals the (relative) volume of $\mathcal{P}$, the second highest coefficient of $E_{\mathcal{P}}(u)$ equals half of the boundary volume of $\mathcal{P}$, and the constant term is $1$.

\begin{thm}{\em (Conjectured by Chapoton)}\label{Arbor-Theorem-3}
The generating series for the Ehrhart polynomials of $Q_{t_n}$ is
\[ 1 + \sum_{n \ge 1} E_{t_n}(u) s^n = \frac{1}{2} \left( 1 - \frac{s}{(us + s - 1)} - \frac{s - 1}{(us + s - 1)^2} \right). \]
\end{thm}
\begin{proof}[First proof of Theorem \ref{Arbor-Theorem-3}]
For every integer $m\geq 1$, Chapoton \cite{Chapoton} introduced the height distribution polynomial $F_{t,m}(X)$ that counts lattice points in the dilated polytope $mQ_{t}$ with weight $X^{h}$ at height $h$, which is multiplicative with respect to direct product of polytopes.

Fix an integer $m\geq 1$ and consider the dilated polytope $mQ_{t_{n}}$.

The arbor $t_n$ consists of a root of size 1 and $n-1$ leaves of size 1. For a single leaf, the height distribution polynomial is $1+X+\dots+X^m$. For $n-1$ independent leaves, their joint distribution is $F_{n-1,m}(X) = (1+X+\dots+X^m)^{n-1} = \sum\limits_{k=0}^{m(n-1)} w_k X^k$.

Obviously, the number of lattice points at height $j$ in $m Q_{t_n}$ is $\sum\limits_{k=0}^{\min(j, m(n-1))} w_k$. Summing over all possible heights $0 \le j \le mn$, the Ehrhart polynomial $E_{t_n}(u)$ evaluated at $m$ gives
\begin{align*}
E_{t_n}(m) &= \sum_{j=0}^{mn}\left(\sum_{k=0}^{\min(j, m(n-1))} w_k\right)=\sum_{k=0}^{m(n-1)} (mn - k + 1) w_k \\
&= (mn+1)\sum_{k} w_k - \sum_{k} k w_k \\
&= (mn+1)F_{n-1,m}(1) - F'_{n-1,m}(1) \\
&= (m+1)^{n-1}\left(\frac{mn}{2} + \frac{m}{2} + 1\right).
\end{align*}
Replacing $m$ with the continuous variable $u$, we construct the generating function $G(s, u) = 1 + \sum\limits_{n\geq1}E_{t_n}(u)s^n$. Then we have
\begin{align*}
G(s, u) &= 1 + \sum_{n\geq1}  (u+1)^{n-1}\left(\frac{un}{2} + \frac{u}{2} + 1\right)s^n\\
&= 1+\frac{\frac{u}{2}+1}{u+1}\sum_{n\geq1}\left((u+1)s\right)^n+\frac{u}{2(u+1)}\sum_{n\geq1}n((u+1)s)^{n}
\\
&=\frac{1}{2}\left(1-\frac{s}{us+s-1}-\frac{s-1}{(us+s-1)^2}\right).
\end{align*}
This completes the proof.
\end{proof}

\begin{proof}[Second proof of Theorem \ref{Arbor-Theorem-3}]
In \cite{Athanasiadis-Xiao-Yan}, the Ehrhart polynomial of the arbor polytope $t_n$ is given by
\begin{align*}
E_{t_n}(u)=\sum_{j=0}^{n-1}(-1)^j \binom{n-1}{j} \binom{(n-j)(u+1)}{n}.
\end{align*}
For the sake of completeness, we briefly outline the proof of this formula due to Athanasiadis et al. \cite{Athanasiadis-Xiao-Yan}.
Observe that $E_{t_n}(u)$ counts the lattice points $(x_1,x_2,\ldots,x_n)\in \mathbb{N}^n$ satisfying
$$x_1+x_2+\cdots +x_n\leq un, \quad\text{and}\quad 0\leq x_i\leq u \quad \text{for} \quad 1\leq i\leq n-1.$$
Since the number of points $(x_1,x_2,\ldots,x_n)\in \mathbb{N}^n$ which satisfy the first inequality and violate the second for any $j$ given values of $i\in [k]$ is equal to $\binom{(n-j)(u+1)}{n}$.
By the principle of inclusion–exclusion, we therefore obtain
\begin{align*}
E_{t_n}(u)=\sum_{j=0}^{n-1}(-1)^j \binom{n-1}{j} \binom{(n-j)(u+1)}{n},
\end{align*}
as desired.

Let $[x^m]f(x)$ be defined as the coefficient of $x^m$ in the polynomial $f(x)$.
Set $a=u+1$, we obtain
\begin{align*}
E_{t_n}(u)&=\sum_{j=0}^{n-1}(-1)^j \binom{n-1}{j} \binom{(n-j)a}{n}=\sum_{j=0}^{n-1}(-1)^j\binom{n-1}{j} [x^n] (1+x)^{a(n-j)}
\\ &=[x^n] (1+x)^{an} \sum_{j=0}^{n-1}\binom{n-1}{j}(-(1+x)^{-a})^j
\\ &=[x^n] (1+x)^{an}(1-(1+x)^{-a})^{n-1}=[x^n] (1+x)^a ((1+x)^a-1)^{n-1}
\\ &=[x^n] (1+x)^a x^{n-1} (g(x))^{n-1}=[x] (1+x)^a (g(x))^{n-1},
\end{align*}
where $g(x)=\frac{(1+x)^a-1}{x}$. Therefore, we have
\begin{align*}
1+\sum_{n=1}^{\infty} E_{t_n}(u) s^n=1+s \sum_{n=1}^\infty [x] (1+x)^a (s\cdot g(x))^{n-1}=1+s[x]\frac{(1+x)^a}{1-s\cdot g(x)}.
\end{align*}
Let $G(x)=\frac{(1+x)^a}{1-s\cdot g(x)}$ and $h(x)=1-s\cdot g(x)$. We have $[x]G(x)=G^{\prime}(0)$.
According to 
$$g(0)=a,\quad g^{\prime}(0)=\frac{a(a-1)}{2},\quad h(0)=1-as,\quad h^{\prime}(0)=-sg^{\prime}(0)=\frac{-sa(a-1)}{2},$$
we get
$$G^{\prime}(0)=\frac{a\cdot h(0)-h^{\prime}(0)}{(h(0))^2}=\frac{a-\frac{a(a+1)s}{2}}{(1-as)^2}.$$
Hence we conclude that
\begin{align*}
1+\sum_{n=1}^{\infty} E_{t_n}(u) s^n &=1+s\cdot G^{\prime}(0)=\frac{1-as+\frac{a(a-1)s^2}{2}}{(1-as)^2}
\\ &= \frac{1}{2}\left( 1-\frac{s}{as-1}-\frac{s-1}{(as-1)^2}\right)
\\ & =\frac{1}{2} \left( 1 - \frac{s}{(us + s - 1)} - \frac{s - 1}{(us + s - 1)^2} \right).
\end{align*}
This completes the proof.
\end{proof}

\subsection{Laplace Transform of Volume Functions}

For a given lattice polytope $\mathcal{P}$, we know that the leading coefficient of $E_{\mathcal{P}}(u)$ equals the volume of $\mathcal{P}$.
For an arbor $t$, Chapoton refine the volume of $Q_t$ into a volume function $V_t(h)$, that records the volume of every slice of $Q_t$ according to the height function $\mathrm{ht}$. This function vanishes outside the segment $[0,n]$, where $n$ is the size of $t$.
For example, for the arbor $t$ on one element, $V_t(h)=1$ for $0\leq h\leq 1$ and $0$ otherwise.

The \emph{Laplace transform} of volume function $V_t(h)$ is defined by 
\begin{align*}
L_t(v)=\int_0^{\infty} \mathrm{exp}(-vh) V_t(h) dh.
\end{align*}
Note that $L_t(0)$ is the volume of $Q_t$.

The operator $\mathbf{T}_n$ acts on a Laplace transform function as follows:
$$\mathbf{T}_{n}\left(\int_0^{\infty} \mathrm{exp}(-vh) f(h) dh\right):=\int_0^{n} \mathrm{exp}(-vh) f(h) dh.$$
Chapoton \cite[Section 2.3]{Chapoton} introduced a bivariate polynomial $L_t(E,V)$ for computing the Laplace transform $L_t(v)$ where $E=e^{-v}$ and $V=1/v$.
For a given arbor $t$, one can compute $L_t(v)$ by first computing polynomial $L_t(E, V)$, and then replace $E$ by $e^{-v}$ and $V$ by $1/v$. Chapoton also derived a formula which tranlates the action of $\mathbf{T}_n$ on polynomials in $E$ and $V$.

\begin{prop}{\em \cite[Lemma 2.5]{Chapoton}}\label{OperatorT_n}
If $l\geq n$, then $\mathbf{T}_{n}(V^{k}E^{l})=0$. Othewise
$$\mathbf{T}_{n}(V^{k+1}E^{l})=V^{k+1}E^{l}-\sum_{j=0}^{k}\frac{(n-l)^{k-j}}{(k-j)!}V^{j+1}E^{n}.$$
\end{prop}

The initial value $L_t(E, V)$ for the arbor of size 1 is $-VE + V$.
Let $t$ be an arbor on the set $[n]$ with $n \geq 2$. 
Let $(t^{(1)}, \ldots, t^{(w)})$ be the sub-trees of the root vertex of $t$. 
Let $r$ be the size of the root vertex of $t$.
The polynomial $L_t(E, V)$ is 
\begin{align}\label{Proposition-2.6}
\mathbf{T}_n \left( \prod_{k=1}^w L_{t^{(k)}}(E, V) \cdot \mathbf{T}_n(V^r) \right).
\end{align}

\begin{thm}{\em (Conjectured by Chapoton)}\label{Arbor-Theorem-4}
The generating series for the Laplace transforms of volume functions of $Q_{t_n}$ is
\[ \sum_{n \ge 1} L_{t_n}(E,V) s^n = V \left( \frac{s}{EVs - Vs + 1} + \frac{Es}{Es - 1} \right). \]
\end{thm}
\begin{proof}
By Equation \eqref{Proposition-2.6}, the Laplace transform of the volume function follows the recursion
$$L_{t_n}(E,V) = \mathbf{T}_n \left( L_{t_1}^{n-1} \cdot \mathbf{T}_n(V) \right),$$
where $L_{t_1}(E,V) = V(1-E)$. According to Proposition \ref{OperatorT_n}, $\mathbf{T}_n(V) = V - VE^n$. Substituting this into the recursion gives
$$L_{t_n}(E,V) = \mathbf{T}_n \left( V^n(1-E)^{n-1} - V^n(1-E)^{n-1}E^n \right).$$
The truncation operator $\mathbf{T}_n$ maps $V^k E^l \mapsto 0$ for $l \ge n$. Since the expanded terms of $V^n(1-E)^{n-1}E^n$ all possess an $E$ degree of at least $n$, $\mathbf{T}_n$ annihilates this entire section. Thus, we only need to evaluate $\mathbf{T}_n \left( V^n(1-E)^{n-1} \right)$.
Using the binomial expansion and Proposition \ref{OperatorT_n}, we obtain
\begin{align*}
L_{t_n}(E,V) &= \mathbf{T}_n \left( \sum_{l=0}^{n-1} \binom{n-1}{l} (-1)^l V^n E^l \right) \\
&=\sum_{l=0}^{n-1} \binom{n-1}{l} (-1)^l\left(V^{n}E^{l}-\sum_{j=0}^{n-1}\frac{(n-l)^{n-1-j}}{(n-1-j)!}V^{j+1}E^{n}\right)\\
&= V^n(1-E)^{n-1} - E^n \sum_{j=0}^{n-1} V^{j+1} \left[ \sum_{l=0}^{n-1} \binom{n-1}{l} (-1)^l \frac{(n-l)^{n-1-j}}{(n-1-j)!} \right].
\end{align*}
The inner bracketed sum corresponds to the $(n-1)$-th order forward difference of the polynomial $P(x) = \frac{x^{n-1-j}}{(n-1-j)!}$ evaluated at $x=1$, denoted $\Delta^{n-1}P(1)$. Since the degree of $P(x)$ is $n-1-j$, the difference $\Delta^{n-1}P(1)$ is strictly $0$ for $j > 0$. When $j=0$, the degree is exactly $n-1$, and $\Delta^{n-1}P(1) = 1$. Consequently, the double sum collapses to a single term
$$L_{t_n}(E,V) = V^n(1-E)^{n-1} - VE^n.$$
Generating the formal power series over $n$:
\begin{align*}
G(s) &= \sum_{n\geq 1} \left( V^n(1-E)^{n-1} - VE^n \right) s^n \\
&= Vs \sum_{n\geq1} (Vs(1-E))^{n-1} - V \sum_{n\geq1} (Es)^n \\
&= V \left( \frac{s}{EVs - Vs + 1} + \frac{Es}{Es - 1} \right).
\end{align*}
This completes the proof.
\end{proof}

%\noindent\textbf{Ethical approval:} Not applicable.

%\noindent\textbf{Conflict interest:} None.

%\noindent\textbf{Declaration of competing interest:} The authors declared that they had no conflicts of interest with respect to %their authorship or the publication of this article.

%\noindent\textbf{Data availability:} Data availability is not applicable to this article as no new data were created or analyzed in %this study.

%\noindent\textbf{Funding:} This work was partially supported by the National Natural Science Foundation of China***********.

\noindent
{\small \textbf{Acknowledgments:}}
%The authors would like to thank the anonymous referees for valuable suggestions for improving the presentation.
The authors are greatly indebted to Professor Guoce Xin for the guidance over the past years. 
This work is partially supported by the National Natural Science Foundation of China [12571355].

\end{document}